\documentclass[12pt,a4paper,reqno]{amsart}
\usepackage{amsmath}
\usepackage{amssymb}
\usepackage{amsthm}
\usepackage{graphicx}
\usepackage{enumerate}
\usepackage[mathscr]{eucal}
\graphicspath{ {images/} }
\usepackage{xcolor}

\newtheorem{theorem}{Theorem}[section]
\newtheorem{prop}[theorem]{Proposition}

\newtheorem{lemma}[theorem]{Lemma}

\theoremstyle{definition}
\newtheorem{definition}[theorem]{Definition}
\newtheorem{remark}[theorem]{Remark}

\newtheorem{example}[theorem]{Example}

\numberwithin{equation}{section}
\usepackage[pagewise]{lineno}
\newcommand{\dopu}{{:}\allowbreak\ }

\newcommand{\R}{{\mathbb{R}}}

\newcommand{\N}{{\mathbb{N}}}

\newcommand{\eps}{\varepsilon}

\newcommand{\loglike}[1]{\mathop{\rm #1}\nolimits}

\newcommand{\lin}{\loglike{lin}}

\newcommand{\diam}{\loglike{diam}}
\newcommand{\1}{{\bf 1}}

\newcommand{\Far}{\loglike{Far}}

\newcommand{\bea}{\begin{eqnarray*}}
\newcommand{\eea}{\end{eqnarray*}}
\newcommand{\beq}{\begin{equation}}
\newcommand{\eeq}{\end{equation}}

\begin{document}

\title[Chebyshev centers that are not farthest points]{Chebyshev centers that are not farthest points}

\author[Sain]{Debmalya Sain}
\author[Kadets]{Vladimir Kadets}
\author[Paul]{Kallol Paul}
\author[Ray]{Anubhab Ray}

\newcommand{\acr}{\newline\indent}

\address[Sain]{Saranagata\\ Dhandighi\\ Contai\\ Purba Medinipur\\ West Bengal\\ India\\ Pin 721401}
\email{saindebmalya@gmail.com}

\address[Kadets]{Department of Mathematics and Informatics\\ Kharkiv V.N. Karazin National University\\ 61022 Kharkiv, Ukraine}
\email{v.kateds@karazin.ua}

\address[Paul]{Department of Mathematics\\ Jadavpur University\\ Kolkata 700032\\ West Bengal\\ INDIA}
\email{kalloldada@gmail.com}

\address[Ray]{Department of Mathematics\\ Jadavpur University\\ Kolkata 700032\\ West Bengal\\ INDIA}
\email{anubhab.jumath@gmail.com}

\thanks{The first author lovingly acknowledges the constant support and encouragement coming from his beloved younger brother Mr. Debdoot Sain. The research of the second author is done in frames of Ukrainian Ministry of Science and Education Research Program 0115U000481; partially it was done during his stay in  Murcia (Spain) under the support of MINECO/FEDER project MTM2014-57838-C2-1-P and Fundaci\'on S\'eneca, Regi\'on de Murcia  grant 19368/PI/14, and partially during his visit to the University of Granada which was supported by MINECO/FEDER project MTM2015-65020-P.
The fourth author would like to thank DST, Govt. of India, for the financial support in the form of doctoral fellowship}


\begin{abstract}
In this paper we address the question whether in a given Banach space, a Chebyshev center of a nonempty bounded subset can be a farthest point of the set. Our exploration reveals that the answer depends on the convexity properties of the Banach space. We obtain a characterization of two-dimensional real strictly convex spaces in terms of Chebyshev center not contributing to the set of farthest points. We explore the scenario in uniformly convex Banach spaces and further study the roles played by centerability and M-compactness in the scheme of things to obtain a step by step characterization of strictly convex Banach spaces. We also illustrate with examples the optimality of our results.

\end{abstract}

\maketitle

\section{Introduction.} 
In this paper letter  $X$  denotes a Banach space, 
$ B_{X}=\{ x\in X ~:~\|x\| \leq 1 \} $ and
$ S_{ X }=\{ x\in X  ~:~\|x\|=1 \} $ denote the unit
ball and the unit sphere of $X$ respectively; $B[x,r]=\{ y\in X  ~:~\|x-y\| \leq r \}$ is the closed ball with center $x$ and radius $ r $ and $ S[x,r]= \{ y\in X  ~:~\|x-y\| = r \}$ is the closed sphere with center $ x $ and radius $r$.
For a set $ A, $ $|A|$ denotes the cardinality of $ A, $ if $ A $ is finite then $ |A| $ is the number of elements in $A.$ We call a subset $A$ of $X$ \emph{nontrivial} if $|A| \ge 2$. For a nonempty bounded set
$ A \subseteq  X$, its \emph{diameter}
is 
$$ 
\diam(A) =\sup \limits_{a_{1},a_{2} \in A}\|a_{1}-a_{2}\|. 
$$
The \emph{outer radius} of $ A \subseteq  X$ at an element $x \in X$
is defined as 
\[ r(x,A) = \sup \limits_{a \in A}  \|x-a\| .\]
The supremum in the definition of $r(x,A)$ may be or may be not attained
at some point of $A$. Let 
\[ F(x,A) = \{ a \in A : \|x-a\| = r(x,A)\} \]
denote the collection of all elements in $ A $ which are farthest from $ x \in X$. If for an element $ x \in X, $ $ r(x,A) $ is not attained then $ F(x,A) = \emptyset$.  
The collection of all elements in $A$ at which $r(x,A)$ is attained for some $ x \in  X $ is denoted by $\Far{A} $ i.e.,
\[ \Far{A} = \bigcup_{x \in  X }  F(x,A) .\]

Recall that the most intriguing unsolved problem about farthest points \cite{Klee} is whether there exists a nontrivial bounded convex closed subset $A$ of a Hilbert space $H$ with the property that  $ |F(x,A)| = 1$ for every $x \in H$ (see also \cite{Asplund} and \cite{SPR}). 

The  \emph{Chebyshev radius} $r(A)$ of $A$ is given by
$ r(A)= \inf \limits_{x \in X } r(x,A). $ If there
exists a point $ c \in  X  $ such that $ r(c,A)=r(A)$,
then $c$ is called a \emph{Chebyshev center} of $A$.  Garkavi \cite{gark} proved 
 that if $X$ is 1-complemented in $X^{**}$ (in particular, if $X$ is reflexive) then every
 bounded subset $A$ of $X$ has a  Chebyshev center, and if $X$ is uniformly convex in every direction, then the  Chebyshev center is unique (see also \cite[Ch. 2, Notes and remarks]{Diest}). Consequently, in uniformly convex spaces, every bounded subset $A$ has a unique Chebyshev center \cite[Part 5 \S 33]{H}.

It is possible to characterize inner product spaces among normed
linear spaces, using the notion of Chebyshev center \cite{BP}.
Let $ c_{A} $ denote a Chebyshev center of a nontrivial bounded subset $A$ of a Banach space $  X . $  In \cite{BP}, Baronti and Papini proved the following inequality for any nonempty subset $ A $ of a Hilbert space $H$:
\[ r^2(x,A) \geq r^2(A) + \|x-c_{A}\|^2 ~~~~~~ \mbox{for all} ~~~~~~ x \in  H , \]
in particular,
\[ r(x,A) > \|x-c_{A}\| ~~~~~~ \mbox{for all} ~~~~~~ x \in  H, \]
for any nontrivial bounded subset $ A $ of $ H. $ It clearly follows from the above inequality that in a Hilbert space $  H$, $ c_{A} \notin \Far{A}$, where $ c_{A} $ is the unique Chebyshev
center of a nontrivial bounded subset $A$ of $H$.

A Banach space $X$ is said to be \emph{strictly convex} if $S_X$ does not contain nontrivial linear segment i.e., there does not exist $ u,v \in S_X~(u \neq v) $ such that  $ \{tu + (1-t)v \dopu t \in [0,1]\} \subset S_X $. Equivalently, $X$ is strictly convex if every  $x \in S_X$ is an extreme point of $B_X$. One more reformulation: $X$ is strictly convex if and only if for every two points $x, y \in X \setminus \{0\}$ with $x \notin \{ty \dopu t > 0\},$ the  \emph{strict triangle inequality} $\|x + y\| < \|x\| + \|y\|$  holds true.

It is clear that if the unit sphere of a Banach space $X$ contains a nontrivial line segment $L = \{tu + (1-t)v \dopu t \in [0,1]\}$ (i.e., $X$ is not strictly convex), then all the points of $A$ are of the same distance 1 from the origin, so $ A = \Far{A}$ and in particular, the Chebyshev center $\frac{(u + v)}{2}$ belongs to $\Far{A}$. This observation motivated Debmalya Sain to ask in ``Research Gate'' the following question:

\vspace{3mm}
Can a Chebyshev center of a bounded set be a farthest point of the set from a point in a strictly convex Banach space?
\vspace{3mm}

This question, which we answer in positive,  leaded to other natural questions and answers, and all these resulted in the article which we are presenting now.
We are indebted to the ``Research Gate'' platform that brought the authors of this paper together.

\smallskip

As we will see in this paper, whether the Chebyshev center of a nontrivial subset of a Banach space may belong to the set, is an important factor in determining the convexity properties of the space. In view of the discussions above, let us introduce the following definitions:
\begin{definition}
A set $A$ in a Banach space $X$ is said to be a \emph{CCF  set} (comes from Chebyshev center in $ \Far A $)  if there is a Chebyshev center of $A$ that belongs to $\Far A$. $A$ is said to be a \emph{CCNF set} (comes from Chebyshev center not in $ \Far A $) if it is not a CCF set.
\end{definition}
\begin{definition}
A Banach space $X$ is said to be \emph{CCF} if it contains a nontrivial CCF set. $X$ is said to be \emph{CCNF} if it is not CCF, i.e., all nontrivial subsets of $X$ are CCNF. 
\end{definition}
 The main results of the paper deal with the general properties of CCF and CCNF spaces. These results are collected in the next section, ingeniously called ``Main results''. At first, in Theorem \ref{theo: CCNF2balls}, for every Banach space $X$, we reduce the question whether $X$ is CCNF to the question whether for every $y \in S_X$ and every $ r \in (0,1), $ the Chebyshev radius of the set $B_X \cap B[y, r]$ is strictly smaller than $r$. 
 
From our earlier discussion, it easily follows that every CCNF space must be strictly convex. In Theorem \ref{theorem:Chebyshev}, using Theorem \ref{theo: CCNF2balls} and a geometric lemma,  for two-dimensional spaces we prove the converse result: every two-dimensional strictly convex real Banach space is CCNF.  However, the result no longer holds true if the dimension of the space is greater than two. We give examples, in both finite-dimensional (Example \ref{example:finite-dim}) and infinite-dimensional (Example \ref{example:infinite}) Banach spaces, to illustrate the scenario. 

\smallskip

The infinite-dimensional example  has an interesting additional property that  $ r(A)= \frac{1}{2} \diam(A)$. Recall, a set with this property is called \emph{centerable}. Our Theorem \ref{theorem:uniform} demonstrates impossibility of such examples in uniformly convex spaces: if $A$ is any nontrivial centerable subset of a uniformly convex Banach space $X$,  then $A$ is CCNF. This result implies the following characterization of finite-dimensional strictly convex Banach spaces (Theorem \ref{theorem:finite-strict}):

\noindent A finite-dimensional Banach space $X$ is strictly convex if and only if every nontrivial bounded centerable  subset of $X$ is CCNF. 

\smallskip
 
The notion of M-compactness also plays a vital role in the study of farthest points. A sequence $ \{a_n\} $ in $A$ is said to be  \emph{maximizing} if for some $ x \in  X  $, $ \|x-a_n\| \rightarrow r(x,A)$. A subset $A$ of $X$ is said to be  \emph{M-compact} if every maximizing sequence in $A$ has a subsequence that converge to an element of $A$. In this paper, in Theorem \ref{theorem:strict-M-compact}, we  prove that in a strictly convex Banach space, every nontrivial, bounded, centerable, M-compact set is CCNF. It is also easy to observe that this property characterizes the strict convexity of a Banach space.  

\smallskip

In the last short section, we demonstrate that all $L_p$ spaces, with $p \neq 2,$ differ dramatically from the Hilbert space in the sense of the properties that we consider in this paper. Namely, although, as we mentioned before, Hilbert spaces are CCNF, all non-Hilbert  $L_p$ spaces of dimension greater than two are CCF.


\section{Main results}

Our first goal is to obtain a geometric characterization of CCNF Banach spaces. To this end, we first reduce the CCNF property of a Banach space to subsets of the form ``intersection of the unit ball with a small ball''. The following two lemmas extract the main ideas of the proof.

\begin{lemma}\label{lemma:large-distance}
Let $A$ be a nontrivial bounded subset of $ X$, $x \in \Far{A}$. Then, for every $N > 0$ there is a point $y \in X$ such that $x$ is a farthest point of $A$ from $y$ and  $\| x-y \| > N$.
\end{lemma}
\begin{proof}
According to the definition of $\Far{A}$, there is a $z \in X$ such that $\| x-z \| \ge \| a-z \| $ for all $a \in A$.
Let us demonstrate that for any $ t > 1, $ $x$ is a farthest point of $A$ from $ tz+(1-t)x$. Indeed, for any $a \in A$,
\begin{eqnarray*}
 \|(tz+(1-t)x)-a\| & \leq & \|tz+(1-t)x-z\| + \|z-a\| \\
                     & \leq & (t-1)\|z-x\| + \|z-x\| \\
			 &	= &   t \|z-x\|  =     \|(tz+(1-t)x)-x\|.						
\end{eqnarray*} 
We observe that $\|tz+(1-t)x\| = \|t(z-x) + x\|   \ge t\|z-x\|  - \|x\| \to \infty$ as $t \to \infty$. Consequently, for sufficiently large $t$, the point $y = tz+(1-t)x$ is what we are looking for. 
\end{proof}
\begin{lemma}\label{lemma:2balls-intersection}
Let $A$ be a nontrivial bounded subset of $ X$, containing its Chebyshev center $c_{A} $. Suppose $c_{A} $ is at the same time a farthest point of $A$ from some $ y \in  X$.
Let $  r $ be the Chebyshev radius of $A$ and $ R = \| c_A-y \|$.
Then $ r \leq R$ and the subset $U =  B[c_{A},r] \cap B[y,R]$ has the following properties: 
\begin{enumerate}
\item[\emph{(a)}] $A \subseteq U$.
\item[\emph{(b)}]  The Chebyshev radius of $U$ equals $r$.
\item[\emph{(c)}] $c_{A} $ is a Chebyshev center of $U$.
\item[\emph{(d)}] $c_{A} $ is a farthest point  of $U$ from $y$.
\end{enumerate}
\end{lemma}
\begin{proof}

Inclusions  $A \subseteq B[c_{A},r]$ and 
\beq \label{lemma:2balls-inters-eq2}
 A \subseteq B[y,R]
\eeq
follow from definitions of Chebyshev center and of farthest point respectively. Consequently, (a) is correct. Because of \eqref{lemma:2balls-inters-eq2},  the Chebyshev radius $r$ of $A$ cannot be greater than $R$. 
Property (a) implies  $r(U) \ge r$, and inclusion
\beq \label{lemma:2balls-inters-eq3}
U \subseteq B[c_{A},r]
\eeq 
implies the reverse inequality, which proves (b). Taking (b) into account, we see that \eqref{lemma:2balls-inters-eq3} means (c). Finally, (d) follows from the fact that $c_{A} \in A \subset U$ and from the inclusion $U \subseteq B[y,R]$.
\end{proof}
Now we are ready to prove the following characterization of CCNF Banach spaces.

\begin{theorem} \label{theo: CCNF2balls}
Denote  $r_{t,z}$ the Chebyshev radius of the set $A_{t,z} = B_X \cap B[z, t]$. Then, for a  Banach space $X$ the following three conditions are equivalent:
\begin{enumerate}
\item[(i)] $X$ is a CCNF space; 
\item[(ii)]   for every $z \in S_X$ and every $t \in (0,1],$ the inequality $r_{t,z} <  t$ holds true;
\item[(iii)]  for every $\eps \in (0,1],$ there is a $t_0 \in (0, \eps)$ such that  for every $z \in S_X$ and every $t \in (0, t_0],$ the inequality $r_{t,z} <  t$ holds true.
\end{enumerate}
\end{theorem}
\begin{proof}
(i) $\Rightarrow$ (ii).  As $ A_{t,z} \subseteq B[z, t] $ we have $r_{t,z} \leq t$.
If $r_{t,z} =t, $ then $z$ is a Chebyshev center of $ A_{t,z}$.  At the same time, $z$ is a farthest point of $A_{t,z}$ from the origin,  which contradicts our assumption (i). Consequently, $r_{t,z} < t$.

The implication  (ii) $\Rightarrow$ (iii) is evident, so it remains to prove 
(iii) $\Rightarrow$ (i).  Assume contrary that $X$ is CCF. Then, by definition, there exists a nontrivial bounded subset $A$ of $X$, containing its Chebyshev center $c_{A}$, such that $ c_{A} \in \Far{A}$. Applying Lemma \ref{lemma:large-distance} for a given $N > 0,$ we can find a $y \in X$ such that $c_A$ is a farthest point of $A$ from $y$ and  $R:=\|c_A-y\| > N$. 
   Denote $r$ the Chebyshev radius of $A$. According to  Lemma~\ref{lemma:2balls-intersection},  $ r \leq R$. Denote $t = \frac{r}{R} \in (0,1]$. Consider the set  $U =  B[c_{A},r] \cap B[y,R]$ from Lemma~\ref{lemma:2balls-intersection}. According to (b) of that lemma, $r(U) = r$.
    
  For every $x \in X,$ denote $f(x) = \frac{1}{R} (x - y)$. Observe that   $f(y) = 0$, $\| f(c_A)\| = 1$ and $f$ multiplies all the distances by the same coefficient $\frac{1}{R}$, i.e., $\|f(x_1) - f(x_2)\| = \frac{1}{R}\|x_1 - x_2\|$ for all $x_1, x_2 \in X$. Consequently,  $r(f(U)) =  \frac{r}{R} = t$. On the other hand,
  $$
  f(U)  = B[f(c_A), \frac{r}{R}] \cap B[f(y), 1] =  B[f(c_A), t] \cap B[0, 1] =A_{t, f(c_A)}.
  $$ 
 So, $r_{t,z} = t$ for $z = f(c_A) \in S_X$ and $t = \frac{r}{R} \le \frac{r}{N} \to 0$  as $N \to \infty$. This contradicts our assumption (iii).
\end{proof}

We next prove that in a two-dimensional strictly convex real Banach space $ X, $ every nontrivial bounded subset of $X$ is CCNF. To this end, we need the following lemma:

\begin{lemma}\label{lemma:2-dimension}
Let $X$ be a two-dimensional real Banach space, $ u,v \in S_ X  $
and let the straight line $ l $ that connects $u$ and $v$ does not
contain origin $ \theta. $ Let $ S $ denote the part of $ B_ X$
not containing $ \theta, $ that is cut from $ B_ X  $
by $l$; $ w=\frac{u+v}{2}, $ $ r=\|u-w\|=\|v-w\|=\frac{1}{2}\|u-v\|. $
Then, $ S \subset w + r B_ X$, i.e., the distance of every
point of $ S $ to $w$ does not exceed $r.$ 
\end{lemma}
\begin{proof}
Clearly, it is sufficient to prove that $ \|s-w\| \leq r $ for all $ s \in S. $\\
Let $ w_{t} = (1-t)u+tw, ~~ 0 \leq t \leq 1 $ and $ w_{t'} = (1-t')w+t'v, ~~ 0 \leq t' \leq 1. $
Now,
\begin{eqnarray} \label{eq1}
\nonumber \|u-w_{t}\| + \|w_{t}-w\| & = & \|u-(1-t)u-tw\| + \|(1-t)u+tw-w\| \\
  \nonumber                         & = & t\|u-w\| + (1-t)\|u-w\| \\
													& = & \|u-w\| = r.  
\end{eqnarray}

\begin{figure}[ht]
\centering
\includegraphics{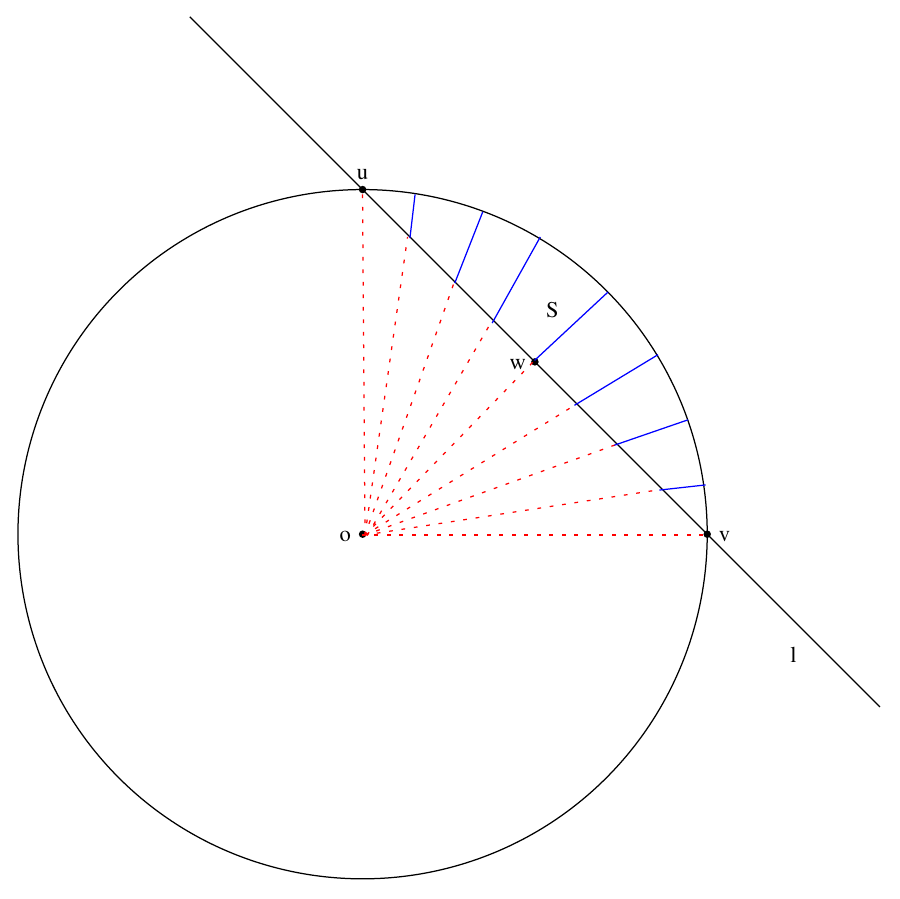} 
\end{figure}

Similarly,
\beq \label{eq2}
\|w-w_{t'}\| + \|w_{t'}-v\| = \|w-v\| = r.  
\eeq
Since $X$ is a two-dimensional real Banach space, for any $ s \in S, $ either $ s=\lambda w_{t} $
or $ s=\lambda w_{t'}, $ for some $ \lambda \geq 1. $ We have,
$ \| \lambda w_{t} \| \leq 1 \Rightarrow \lambda \leq \frac{1}{\|w_{t}\|} $
and also, $ \| \lambda w_{t'} \| \leq 1 \Rightarrow \lambda \leq \frac{1}{\|w_{t'}\|} $. 
Now, 
\begin{eqnarray}\label{eq3}
\nonumber\| \lambda w_{t}-w_{t} \| & = & (\lambda -1)\|w_{t}\| \\
\nonumber                          & \leq &\left (\frac{1}{\|w_{t}\|}-1\right)\|w_{t}\| \\
\nonumber													& = & 1 -\|w_{t}\| \\
													& = & \|u\| -\|w_{t}\| \leq \|u-w_{t}\|. 
\end{eqnarray}
\noindent Similarly, 
\beq \label{eq4}
 \| \lambda w_{t'}-w_{t'}\| \leq \|v-w_{t'}\|.
\eeq
Now using \eqref{eq1} and \eqref{eq3}, we have,
\begin{eqnarray*}
\| \lambda w_{t}-w \| & = & \|\lambda w_{t}-w_{t}+w_{t}-w\| \\
                          & \leq & \|\lambda w_{t}-w_{t}\| + \|w_{t}-w\| \\
													& \leq & \|u-w_{t}\| + \|w_{t}-w\| \\
													& = & \|u-w\| = r. 
\end{eqnarray*}
Similarly, using \eqref{eq2} and \eqref{eq4}, we can show that $ \| \lambda w_{t'}-w\| \leq r. $  So for all $ s \in S, $ $ \|s-w\| \leq r$, which completes the proof.  

\end{proof}

Now we are ready to prove the promised theorem.
\begin{theorem}\label{theorem:Chebyshev}
Let $X$ be a two-dimensional strictly convex real Banach space. Let $A$ be a nontrivial bounded subset of $  X$, containing its Chebyshev center $ c_{A} $.  Then $ A $ is CCNF. 
\end{theorem}
\begin{proof}
We will use the notations of Lemma~\ref{lemma:2balls-intersection}.

\begin{figure}[ht]
\centering 
\includegraphics{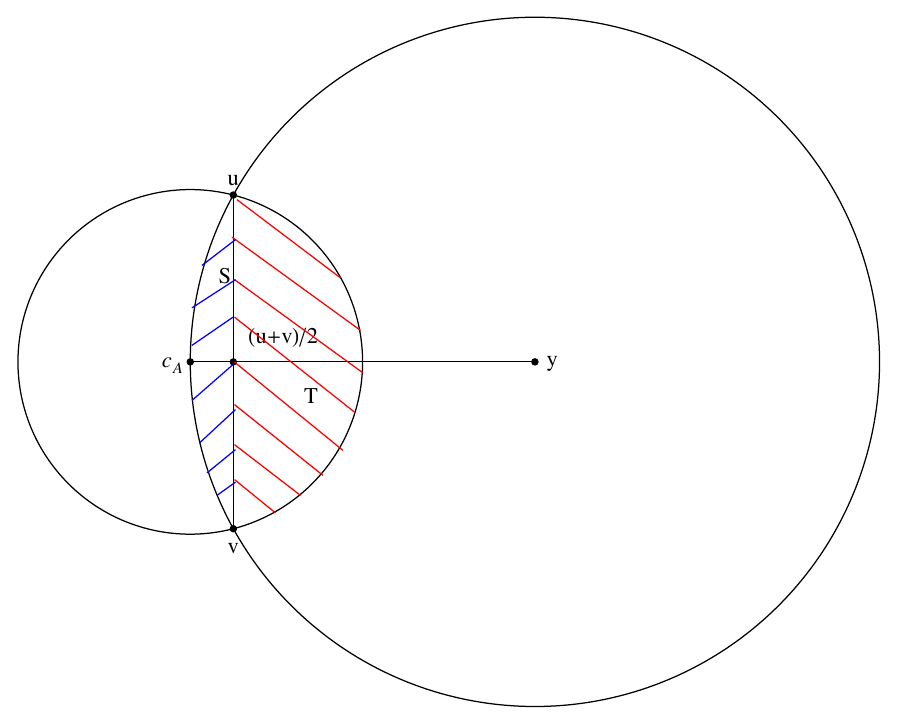}
\end{figure}

Suppose $ c_{A} \in A $ is a farthest point of $ A $ from some $ y \in X. $ Let $ r $
be the Chebyshev radius of $ A $ and $ R = \|x-y\|. $  
Let $ u,v $ be the intersection points of the spheres $ S[c_{A},r] $ and $ S[y,R]$.
Then by Lemma~\ref{lemma:2-dimension}, both $ S $ and $ T $ in the above
picture are subsets of the closed ball centered at $ \frac{(u+v)}{2} $
and radius $ \|\frac{(u-v)}{2}\|. $ Then $ A \subseteq B[\frac{(u+v)}{2},~\|\frac{(u-v)}{2}\|]. $
By the definition of Chebyshev radius$, \|\frac{(u-v)}{2}\| \geq r $
which implies that $ \|u-v\| \geq 2r. $ On the other hand,
$ u,v \in S[c_{A},r]$, so $ \|u-v\| \leq 2r $ and consequently $ \|u-v\|=2r$.
We have $ \|(u-c_{A})+(c_{A}-v)\|=2r $ and $ \|u-c_{A}\|+\|c_{A}-v\|=2r$.
As the space is strictly convex, we must have $ (u-c_{A})=k(c_{A}-v)$, for some constant $ k > 0$.
Since $ \|u-c_{A}\|=\|c_{A}-v\|=r$, we have $ k=1$. Therefore, we have $ c_{A}= \frac{u+v}{2}$. 
Now $ u,v, c_{A} \in S[y,R]$ and so by strict convexity we get 
$ u=v=c_{A} $.  Then $r=0$ and so $A$
consists of only one point, contradicting our assumption that $A$ is nontrivial.
This completes the proof of the theorem.
\end{proof}

The converse of  Theorem~\ref{theorem:Chebyshev} is also true. Indeed, as we already remarked in the introduction, if $X$  is not strictly convex, then $ S_{ X } $ contains a straight line segment $ L=\{(1-t)u+tv : u, v \in S_{ X }, t \in [0, 1]\} $.  It is easy to see that $ \frac{u+v}{2} $ is a Chebyshev center of $ L$, which is also a farthest point of $ L $ from the origin. Thus, we have the following characterization of strict convexity of a two-dimensional real Banach space: 

\begin{theorem}\label{theorem:Chebyshev-char}
A two-dimensional real Banach space $X$ is strictly convex if and only if every nontrivial bounded subset $A$  which contains its Chebyshev center is CCNF.  
\end{theorem}

In general Theorem~\ref{theorem:Chebyshev} is not true if the dimension of the space is strictly greater than two. 
The following two examples illustrate the situation in both finite and infinite-dimensional strictly convex spaces.
Firstly we recall an easy but useful way to construct equivalent strictly convex norms \cite[Ch. 4 \S 2, Theorem 1]{Diest}.
\begin{prop} \label{prop1-str-conf-norm}
Let $X$, $Y$ be Banach spaces, $Y$ be strictly convex and let $T \dopu X \to Y$ be an injective continuous linear operator. For $x \in X, $ denote $p(x) = \|x\| + \|Tx\|$. Then $(X, p)$ is strictly convex.
\end{prop}

\begin{example} \label{example:finite-dim}
Consider $X= (\mathbb{R}^n, \|. \|), n \geq 3 $ where 
$$ \|(x_{1},x_{2},\ldots, x_n)\|=\sum_{i=1}^{n} |x_{i}| + \frac{1}{2}\sqrt{\sum_{i=1}^{n}|x_{i}|^2}.$$
It is easy to see that the norm is of the form given by  Proposition \ref{prop1-str-conf-norm}. So $X$ is strictly convex and by finite-dimensionality, $ X $ is uniformly convex as well. Consequently, for any bounded set, the Chebyshev center is unique.

Let $\{e_1, \ldots, e_n\}$ be the canonical basis of $\R^n$, i.e., $e_1 = (1,0,0,\ldots, 0)$,  $e_2 = (0,1,0,0,\ldots, 0)$, etc. Also denote $\theta = (0,0,\ldots, 0)$.
Let 
$$ 
A = \{\theta,e_1, e_2, \ldots, e_n\}.
$$   
Consider  $ z = (1,1,\ldots, 1) \in \mathbb{R}^n$. Then 
$ \|z -e_k\|= (n-1) + \frac{\sqrt{(n-1)}}{2} $ for all $k = 1, \ldots, n$.
 However,
 $$
 \|z-\theta\|= n + \frac{\sqrt{n}}{2} > (n-1) + \frac{\sqrt{(n-1)}}{2},
 $$
which proves that $\theta$ is the farthest point of $ A $ from $z$.

\smallskip

We claim that $\theta$ is the Chebyshev center of $A$.
If $ (a_{1},a_{2},\ldots, a_{n})\in \mathbb{R}^n $ is a Chebyshev center of $A$,
then by symmetry, all the cyclic permutations of coordinates 
$$
(a_{2},a_{3},\ldots, a_{n},a_{1}), \ldots, (a_{n},a_{1},\ldots, a_{n-1}) 
$$ 
give us Chebyshev centers of $A$ as well. Since the set of all Chebyshev centers is convex, $ (\alpha,\alpha,\ldots, \alpha) $ is also a Chebyshev center of $ A, $ where $ \alpha = \frac{a_{1}+a_{2}+\ldots +a_{n}}{n}. $ By uniqueness of Chebyshev center in uniformly convex spaces, the Chebyshev center of $A$ should be of the form $sz$, $s \in \R$.  As $ \|a-\theta\|=\frac{3}{2}$ for all $ a \in A\setminus \{\theta\}$, it is sufficient to demonstrate that for any $s \in \R,$ there is a  $ p \in A$ such that $\|sz - p\|\geq \frac{3}{2}$.

\smallskip

If $ s \ge 1$ then considering $p = \theta$ we are done. If $ s<0$ then considering  $e_1$ 
 as $ p $ we are also done.
Let $ 0<s<1$. In this case, let us also take $p = e_1$.
We have, 
\begin{eqnarray*}
 2s > s &\Rightarrow& (1-2s) < (1-s) \\
         & \Rightarrow& 1-(2s+2s-2s) < (1-s) \leq \sqrt{(1-s)^2+s^2+\ldots +s^2} \\
         & \Rightarrow& (s+s-s)+ \frac{1}{2}\sqrt{(1-s)^2+s^2+\ldots +s^2} > \frac{1}{2} \\
	   & \Rightarrow& (1-s)+s+\ldots +s+ \frac{1}{2}\sqrt{(1-s)^2+s^2+\ldots +s^2} > \frac{3}{2} \\
	   & \Rightarrow&\|s z - e_1\|  > \frac{3}{2}.
\end{eqnarray*}
This proves that $\theta$ is the Chebyshev center of $A$.
\end{example}

\begin{remark} \label{remark-example:finite}
Applying Lemma \ref{lemma:2balls-intersection} to the set $A$, from Example \ref{example:finite-dim}, we deduce that there exists a finite-dimensional  uniformly convex Banach space $X$ and a non-trivial convex compact set $U \subset X$ such that $U$ is a CCF subset of $ X$.
\end{remark}

Now we present a similar  example with a centerable subset. The example ``lives'' in an infinite-dimensional  strictly convex Banach space $X$. Afterwords, it will follow from Theorem \ref{theorem:uniform} that such an example is impossible in finite-dimensional strictly convex Banach spaces.

\begin{example}\label{example:infinite}
Consider the space $ c_{0} $ of all sequences of real numbers converging to zero, equipped with
the following norm:
\beq \label{example2:eq1}
 \|x\|= \max_{k} |x_{k}| + \sqrt{\sum_{k=1}^{\infty}\frac{1}{4^k}|x_{k}|^2} 
\eeq
where $ x_{k} ~ (k \in \mathbb{N}) $ denote the $ k- $th coordinate of $x \in c_{0}$. 
Clearly, the norm is strictly convex. Let us denote this Banach space by $  X . $ Let $ \theta = (0,0,\ldots, 0, \ldots) $ and $e_{n} = (0,0, \ldots, 0,1,0 \ldots)$, i.e., the n-th coordinate of $ e_{n} $ is $ 1 $ and all other coordinates are $ 0. $ Denote 
$$
x_n = \frac{1}{n} e_{1} + \left(1-\frac{1}{n}\right)e_{n},  \quad y_n = \frac{1}{n}e_{1} -\left(1-\frac{1}{n}\right)e_{n}
$$ 
and consider $ A=\{\theta\} \cup \{x_n  \dopu n=2,3,...\} \cup \{y_n  \dopu n=2,3,...\}$.

 We claim that $A$ is a subset of the unit ball and consequently, $ r(A) \leq 1$. In fact,
$$
\|x_n\|=\|y_n\|=\left(1-\frac{1}{n}\right)+\sqrt{\frac{1}{4n^2}+\frac{1}{4^n}\left(1-\frac{1}{n}\right)^{2}}.
$$ 
Since $ \frac{1}{4^n} \leq \frac{1}{4n^2} $ for all $ n=2,3,4, \ldots $,
$$
\|x_n\|=\|y_n\| \leq \left(1-\frac{1}{n}\right) +\sqrt{\frac{2}{4n^2}}<1-\frac{1}{n}+\frac{1}{n}=1.
$$
The claim is proved.  Now,
$$
\lim\limits_{n \to \infty}\|x_n - y_n\|=\lim\limits_{n \to \infty} 2\left(1-\frac{1}{n}\right)\|e_{n}\| \geq \lim\limits_{n \to \infty}2\left(1-\frac{1}{n}\right)=2. 
$$
Consequently, $  \diam(A) \geq 2. $ Since $r(A) \geq \frac{1}{2} \diam(A)$, we have $ r(A)=1$. So $ \theta $ is a Chebyshev center of $ A$. 
Finally we prove that $ \theta $ is a farthest point of $A$ from $ u=e_{1} $.  In fact,
$ \|e_{1}-\theta\|=\|e_{1}\|=\frac{3}{2} $.  On the other hand,
\begin{eqnarray*}
\| e_{1}-x_n\| & = &\| e_{1}-y_n\|   =\left(1-\frac{1}{n}\right)\|e_{1} \pm e_{n}\|\\
                                                    & = & \left(1-\frac{1}{n}\right)\left(1+ \sqrt{\frac{1}{4}+\frac{1}{4^n}}\right)\\
                                                    & \leq &\left(1-\frac{1}{n}\right)\left(\frac{3}{2}+\frac{1}{2^n}\right)\\
																										& = & \frac{3}{2}+\frac{1}{2^n}-\frac{3}{2n}-\frac{1}{n2^n}  <  \frac{3}{2}.
\end{eqnarray*} 
So $ \theta $ is the farthest point of $A$ from $ e_{1} $. 

\end{example}

Next, we prove that if $A$ is a bounded centerable subset in a uniformly convex Banach space, then $A$ is CCNF. Before doing this, let us recall one of the standard equivalent definitions of uniform convexity: 
a Banach space $X$ is said to be \emph{uniformly convex} if for every two sequences $\{x_n\}$, $\{y_n\}$ in $B_X,$ the condition $\lim\limits_{n \to \infty}\|x_n + y_n\| = 2$ implies $\lim\limits_{n \to \infty}\|x_n - y_n\| = 0$.
  
\begin{theorem}\label{theorem:uniform}
Let $X$ be a uniformly convex Banach space and $A$ be a nontrivial bounded centerable subset of $X,$ containing its Chebyshev center $c_{A}$. Then $ A $ is CCNF.  
\end{theorem}
\begin{proof}
Let $ r > 0$ be the Chebyshev radius of $A$. According to the definition of a centerable set, there are $u_n, v_n \in A$, $n = 1,2, \ldots$ such that 
\beq \label{eq-thm-unconv-eq1}
\lim\limits_{n \to \infty}\|u_n - v_n\| = 2r.
\eeq
 Consider elements 
$$
x_n = \frac{1}{r}(u_n - c_A), \quad y_n = \frac{1}{r}(c_A - v_n). 
$$
Then $x_n, y_n \in B_X$, $\lim_{n \to \infty}\|x_n + y_n\| = 2$, so the uniform convexity of $ X $ implies $\lim_{n \to \infty}\|x_n - y_n\| = 0$. This means that $\lim_{n \to \infty}\|u_n +v_n - 2 c_A\| = 0$. In other words, 
$$
{u_n +v_n} \to 2 c_A.
$$

Suppose $ c_{A} $ is a farthest point of $A$ from  some $ y \in  X$.  Denote $ R = \| c_{A} - y \|$. Now, denote 
$$
\tilde x_n = \frac{1}{R}(u_n - y), \quad \tilde y_n = \frac{1}{R}( v_n - y). 
$$
Then $\tilde x_n, \tilde y_n \in B_X$, 
$$
\lim_{n \to \infty}\|\tilde x_n + \tilde y_n\| = \frac{1}{R} \lim_{n \to \infty}\|u_n + v_n - 2y\| =  \frac{1}{R} \|2 c_A - 2y\|  = 2.
$$

Again, the uniform convexity of $ X $ implies $\lim\limits_{n \to \infty}\|\tilde x_n - \tilde y_n\| = 0$, i.e., $\|u_n - v_n \| \to 0$, which contradicts \eqref{eq-thm-unconv-eq1}.
 This contradiction completes the proof of the theorem. 
\end{proof}

Since in the finite-dimensional case, strict convexity implies uniform convexity, it is possible to obtain the following characterization of finite-dimensional strictly convex Banach spaces, simply by observing that any straight line segment in a Banach space is always a centerable set. 

\begin{theorem}\label{theorem:finite-strict}
A finite-dimensional Banach space $X$ is strictly convex if and only if every nontrivial bounded centerable subset $A$ which contains its Chebyshev center is CCNF.  
\end{theorem}
\begin{remark} 
Example \ref{example:infinite} shows that the uniform convexity condition in Theorem~\ref{theorem:uniform} cannot be substituted by strict convexity.
\end{remark}
In the next theorem, we prove that if $A$ is a bounded centerable M-compact subset in a strictly convex Banach space, then $ A $ is CCNF. Before proving the theorem we first prove the following lemma:\\
\begin{lemma}\label{lemma:attainment}
Let $X$ be a Banach space. Let $ A $ be any nontrivial bounded
centerable M-compact subset of $X,$ containing its Chebyshev center $ c_{A}$.
Then $ A $ attains its diameter.
\end{lemma}
\begin{proof}
Since in our case, $$ 
\diam(A) = \sup \limits_{a,b \in A}  \|a-b\|  = 2r(A), 
$$ 
there exist sequences $ \{ x_{n}\},~\{ y_{n}\} \subset A $ such that $ \|x_{n}-y_{n}\| \rightarrow  2r(A)$.
We claim that $ \{x_{n}\} $ is a maximizing sequence in $ A $ for $ c_{A}$. If not, then there exists $ \eps_{0} > 0 $ and a subsequence $ \{ x_{n_{k}} \} $ such that
 $ \|c_{A}- x_{n_{k}}\| \leq r(A) - \eps_{0}. $
Then,
\begin{eqnarray*}
\| x_{n_{k}}-y_{n_{k}}\| &=& \|( x_{n_{k}}-c_{A}) + (c_{A}-y_{n_{k}})\| \\
                & \leq & \|(x_{n_{k}}-c_{A})\| + \|(c_{A}-y_{n_{k}})\| \\
								& \leq & r(A) - \eps_{0} + r(A) =  2r(A) - \eps_{0},
\end{eqnarray*}
which contradicts the fact that $ \|x_{n}-y_{n}\| \rightarrow 2r(A)$. 
By the same argument, $ \{ y_{n} \} $ is a maximizing sequence in $ A $ for $ c_{A}$. 
Consequently, as $ A $ is M-compact,  there is a subsequence $\{n_{k}\} \subset \N$ and there are $ \tilde x,  \tilde y \in A$ such that    $x_{n_{k}} \to \tilde x$ and  $y_{n_{k}} \to \tilde y$. Then 
$$
\diam(A) = \lim \limits_{n \to \infty} \|x_{n}-y_{n}\| = \lim \limits_{k \to \infty} \|x_{n_{k}}-y_{n_{k}}\| = \| \tilde x -  \tilde y\|. 
 $$ 
Thus diameter of $ A $ is attained. 								
\end{proof} 

We now prove the desired theorem.

\begin{theorem}\label{theorem:strict-M-compact}
Let $X$ be a strictly convex Banach space and $ A \subset X$ be a
nontrivial bounded centerable  M-compact subset,
containing its Chebyshev center $c_{A}$. Then $ A $ is CCNF. 
\end{theorem}
\begin{proof}
Suppose $ A $ is CCF. Then $ c_{A} \in \Far A$. By the definition, there exists $x \in X$
such that $ c_{A} \in F(x,A) $. Denote 
$$
R =  \|x-c_{A}\|= \sup \limits_{a \in A} \|x-a\|.
$$
Due to Lemma~\ref{lemma:attainment},
$ \diam(A)$ is attained and since $ A $ is centerable, $\diam(A)=2r(A)$. This means that there exist $ a_{1},a_{2} \in A $ such that
\begin{equation} \label{eq-diam2rattained}
  \|a_{1}-a_{2}\|= \sup \limits_{a,b \in A}\|a-b\| =2r(A).
\end{equation}
We claim that $ \|c_{A}-a_{1}\|=\|c_{A}-a_{2}\|=r(A)$. Clearly $ \|c_{A}-a_{1}\| \leq r(A)$ and $\|c_{A}-a_{2}\|\leq r(A)$.
Moreover, the assumption that one of them is strictly smaller than $r(A)$ leads to a contradiction: 
$$
2r(A) =  \|a_{1}-a_{2}\|=\|a_{1}-c_{A}+c_{A}-a_{2}\| \leq \|a_{1}-c_{A}\|+\|a_{2}-c_{A}\|< 2r(A).
$$
So, the claim is proved.  Now, 
$$
\left\|\frac12\left((a_{1}-c_{A})+(c_{A}-a_{2})\right)\right\| = r(A).
 $$ 
 Geometrically this means that $a_{1}-c_{A}, c_{A}-a_{2}$ and $\frac12\left((a_{1}-c_{A})+(c_{A}-a_{2})\right)$ belong to the same sphere $r(A)S_X$. 
By the strict convexity of $ X, $ it follows that $a_{1}-c_{A} =  c_{A}-a_{2}$, i.e., $c_A = \frac12 (a_{1}+a_{2})$.
The following chain of inequalities
\begin{eqnarray*}
R &=&\|x-c_{A}\| = \bigl\|\frac12 \bigl((x-a_{1})+(x-a_{2}) \bigr)\bigr\|  \\
            &\le& \frac12  \bigl\|x-a_{1}\bigr\| +   \frac12 \bigl\|x-a_{2}\bigr\| \le \sup \limits_{a \in A} \|x-a\|  = R
\end{eqnarray*}
implies that all of them are equalities, i.e., all three vectors $x-a_{1}$, $x-a_{2}$, and $\frac12 \bigl((x-a_{1})+(x-a_{2}) \bigr)$  belong to the same sphere $R S_X$. Then, the strict convexity of $ X $ implies that $x-a_{1} = x-a_{2}$, i.e., $a_1 = a_2$. This contradiction with \eqref{eq-diam2rattained} completes the proof of the theorem. 
\end{proof}
\begin{remark}
Example \ref{example:infinite} shows that the M-compactness condition in Theorem \ref{theorem:strict-M-compact} can not be removed.
\end{remark}

Now, we can give an characterization of strictly convex Banach spaces, simply by observing that any closed straight line segment in a Banach space is always a centerable and M-compact set. Thus, we have the following theorem : 
\begin{theorem}\label{theorem:strictly-convex}
A Banach space $X$ is strictly convex if and only if every nontrivial bounded centerable and M-compact subset $A \subset X$ which contains its Chebyshev center is CCNF.  
\end{theorem}
\begin{remark}
Theorem \ref{theorem:strict-M-compact} shows that the uniform convexity condition in Theorem \ref{theorem:uniform} can be substituted by strict convexity if we impose an additional condition of M-compactness on the subset $A$ of $X$.
\end{remark}
We would like to add a final comment that Theorem \ref{theorem:Chebyshev-char}, Theorem \ref{theorem:finite-strict} and Theorem \ref{theorem:strictly-convex} together yield a nice step by step characterization of strict convexity of a Banach space. The characterizing properties follow an interesting trend, depending on the dimension of the space. Accordingly, we state the following theorem as the final result of this section:
\begin{theorem}
Let $X$ be a Banach space. Then the following holds. \acr
\emph{(a)} If $X$ is a two-dimensional real Banach space, then $X$ is strictly convex if and only if every nontrivial bounded subset $A$  which contains its Chebyshev center is CCNF. \acr
\emph{(b)} If $X$ is a finite-dimensional Banach space, then $X$ is strictly convex if and only if every nontrivial bounded centerable subset $A$ which contains its Chebyshev center is CCNF. \acr
\emph{(c)} If $X$ is any Banach space, then $X$ is strictly convex if and only if every nontrivial bounded centerable and M-compact subset $A$ which contains its Chebyshev center is CCNF.
\end{theorem}


\section{Chebyshev centers in $L_p$ spaces}

In this section we demonstrate that all Banach spaces $L_p$, $p \neq 2$, of dimension greater than two are CCF. Since $L_1$ and $L_\infty$ are not strictly convex, for them this result follows from the previous discussion. So, in this section we consider only $1 < p < \infty$. We begin with an elementary technical proposition in dimension three. Let $\ell_p^{(3)}$ denote the space $\R^3$ equipped with the norm $\|(x_1,x_2, x_3)\| = (|x_1|^p + |x_2|^p + |x_3|^p)^{1/p}$, and let $e_1= (1,0,0)$, $e_2= (0,1,0)$, and $e_3= (0,0,1)$.

\begin{prop} \label{prop:ellp3-1}
The Chebyshev center of the set $A_0 = \{e_1, e_2, e_3\} \subset \ell_p^{(3)}$ is the point $x_p = (s_p, s_p, s_p)$, where 
$$
s_p = \frac{1}{1 + 2^{1/(p-1)}}.
$$
\end{prop}
\begin{proof}
Since $\ell_p^{(3)}$ is uniformly convex, $A_0$ possesses unique Chebyshev center and by symmetry, this Chebyshev center must be of the form $(s, s, s)$.  What remains to do, is to minimize the quantity
$$
f(s) = \|e_k - (s, s, s)\|^p = |1 - s|^{p} + 2 |s|^p, s \in \R.
$$
Evidently, the minimum attains on $(0, 1)$ (otherwise $f(s) \ge 1$), where $f'(s) = 2ps^{p-1} - p(1-s)^{p-1}$, and $s_p$ is the unique root of equation $f'(s) = 0$.
\end{proof}

Following the notation of the previous proposition, denote $A_p = \{e_1, e_2, e_3, x_p\} \subset \ell_p^{(3)}$. 
\begin{prop} \label{prop:ellp3-2}
For $p \in (1,2) \cup (2, \infty),$ $A_p$ is a CCF set and consequently, $\ell_p^{(3)}$ is a CCF space.
\end{prop}
\begin{proof}
$A_p$ is formed by $A_0$ together with its Chebyshev center $x_p$, so  $x_p$ is also the  Chebyshev center of $A_p$. It remains to show that $x_p \in \Far{A_p}$. We consider the following two cases separately:

Case 1: $p \in (1,2)$. In this case 
\begin{equation} \label{eq:sp<13}
0 < s_p  < \frac13.
\end{equation}
We are going to demonstrate that for $t > 1$ large enough, $x_p$ is the farthest point of $A_p$ from $y = (t,t,t)$. The distance from $y$ to any of $e_k$ equals $ ((t-1)^p + 2t^p)^{1/p}$, $\|y - x_p\| =  3^{1/p} (t - s_p)$, so we need to check for large $t$ the inequality
$$
(t-1)^p + 2t^p < 3\left(t - s_p \right)^p.
$$
Dividing by $t^p$ and denoting $\tau = \frac1t,$ we reduce this to 
\begin{equation} \label{eq:tauuup}
(1 - \tau)^p + 2 < 3(1 - s_p\tau)^p
\end{equation}
for small positive $\tau$. At the point $\tau = 0,$ the left-hand side of \eqref{eq:tauuup} equals the right-hand side. So in order to demonstrate \eqref{eq:tauuup} for $\tau$ close to $0,$ it is sufficient to show for  $f_1(\tau) = (1 - \tau)^p + 2$, $f_2(\tau) =  3(1 - s_p\tau)^p,$ the validity of the inequality  $f_1'(0)  < f_2'(0) $. This is the inequality
$$
- p < -3p  s_p,
$$
which follows from  \eqref{eq:sp<13}.

Case 2: $p \in (2, \infty)$. In this case 
\begin{equation} \label{eq:sp>13}
s_p  > \frac13.
\end{equation}
We are going to demonstrate that for $t > 0$  large enough, $x_p$ is the farthest point of $A_p$ from $y = (-t,-t,-t)$. The distance from $y$ to any of $e_k$ equals $ ((t+1)^p + 2t^p)^{1/p}$, $\|y - x_p\| =  3^{1/p} (t + s_p)$, so we need to check for large $t$ the inequality
$$
(t+1)^p + 2t^p < 3\left(t + s_p \right)^p.
$$
The same way as above, this reduces to 
\begin{equation*} \label{eq:tauuupcase2}
(1 + \tau)^p + 2 < 3(1 + s_p\tau)^p
\end{equation*}
for small positive $\tau$. Denoting $g_1(\tau) = (1 + \tau)^p + 2$, $g_2(\tau) =  3(1 + s_p\tau)^p$, we have to demonstrate the inequality  $g_1'(0)  < g_2'(0)$, i.e., the inequality
$$
p < 3p  s_p,
$$
which follows from  \eqref{eq:sp>13}.
\end{proof}

\begin{theorem}\label{theorem:L_P}
Let $(\Omega, \Sigma, \mu)$ be a finite or $\sigma$-finite measure space, containing a disjoint triple $\{\Delta_i\}_{i = 1}^3 \subset \Sigma$ of subsets of finite positive measure. Then $L_p = L_p (\Omega, \Sigma, \mu)$ is a CCF space for every  $p \in (1,2) \cup (2, \infty)$.
\end{theorem}
\begin{proof}
Denote $f_i = \1_{\Delta_i}/\|\1_{\Delta_i}\|$, $i = 1,2,3$, $E = \lin\{f_i\}_{i = 1}^3 \subset L_p$. It is well-known (and can be checked easily) that $E$ is isometric to $\ell_p^{(3)}$, where the corresponding isometry $T: \ell_p^{(3)} \to E$ acts as follows: 
$$
T(x_1,x_2,x_3) = x_1f_1 + x_2f_2 + x_3f_3.
$$
It is also well-known that $E$ is 1-complemented in $L_p$ with the corresponding projection  $P: L_p  \to E$ being
$$
Pf = \sum_{i=1}^3 \frac{\int_{\Delta_i} f d\mu}{\mu(\Delta_i)}  \1_{\Delta_i}.
$$
(Equality $\|P\| = 1$ follows from H\"older's inequality). 

  Let $A_p$ be the set from Proposition \ref{prop:ellp3-2}.  If we consider  $T(A_p)$ as a subset of $E$, then $Tx_p$ is its Chebyshev center, because $T$ is an isometry. Since $E$ is 1-complemented in $L_p$, $Tx_p$ is also the Chebyshev center of $T(A_p)$, when  $T(A_p)$  is considered as a subset of $L_p$. Let  $y \in \ell_p^{(3)}$ be such a point that $x_p \in F(y,A_p)$. Since $T$ is an isometry,  $Tx_p$ is the farthest point in  $T(A_p)$ from $Ty$. This means that the Chebyshev center $Tx_p$ of $T(A_p) \subset L_p$ is a  farthest point.
\end{proof}


\begin{thebibliography}{99}

\bibitem{Asplund} Asplund, E., \textit{Sets with unique farthest points}, Israel Journal of Mathematics, \textbf{5} (1967), 201--209.


\bibitem{BP} Baronti, M. and Papini, P. L.,
  \textit{Remotal sets revisited},
  Taiwanese Journal of Mathematics,
  \textbf{5} (2001), 367--373.

\bibitem{Diest} Diestel, Joseph.,
\textit{Geometry of Banach spaces. Selected topics}. 
Lecture Notes in Mathematics, \textbf{485},  XI, 282 p.  (1975).
	
\bibitem{gark} Garkavi, A. L.,  \textit{On the optimal net and best cross-section of a set in a normed space}. (Russian)
Izv. Akad. Nauk SSSR Ser. Math.   \textbf{26}, No. 1, 87 -- 106 (1962); MR0136969. Translated in:  Garkavi, A. L. The best possible net and the best possible cross-section of a set in a normed space. Amer. Math. Soc. Transl., Ser. 2, 39, 111 -- 132
(1964).
	
\bibitem{H} Holmes,R.B., 
	\textit{A course on optimization and best approximation},
 Lecture notes in  Mathematics, \textbf{257},  VIII, 233 p.  (1972).

\bibitem{Klee}  Klee, V., \textit{Convexity of Chebyshev sets}, Mathematische Annalen,
\textbf{142 } (1961), 169--178.


\bibitem{SPR} Sain, D., Paul, K. and Ray, A.,
\textit{Farthest point problem and M-compact sets}, arXiv:1605.04100v1[math.FA], 13 May 2016.


	
\end{thebibliography}
\end{document}